\documentclass[12pt]{article}
\usepackage{amssymb}
\usepackage{graphicx}

\usepackage{color}
\usepackage{latexsym}
\usepackage{amssymb,amsmath,amstext}
\usepackage{epsfig}
\usepackage{graphics}
\usepackage{subfigure}
\usepackage{euscript}
\usepackage{ifthen}
\usepackage{wrapfig}
\usepackage{amsthm}

\theoremstyle{plain}%
\newtheorem{theorem}{Theorem}

\newtheorem{example}[theorem]{Example}
\newtheorem{lemma}[theorem]{Lemma}
\newtheorem{corollary}[theorem]{Corollary}
\newtheorem{definition}[theorem]{Definition}
\newtheorem{remark}[theorem]{Remark}
\newtheorem{conjecture}[theorem]{Conjecture}

\date{}
\begin{document}

\title{Monomials, Binomials and Riemann-Roch}
\author{Madhusudan Manjunath and Bernd Sturmfels}

\maketitle

\begin{abstract}
\noindent
The Riemann-Roch theorem on a graph $G$ is  related to Alexander duality in combinatorial commutive algebra.
We study the lattice ideal given by chip firing on $G$ and the initial ideal whose standard monomials are the
$G$-parking functions. When $G$ is a saturated graph,  these ideals are generic and the Scarf complex is a minimal free resolution.
Otherwise, syzygies are obtained by degeneration. We also develop a self-contained Riemann-Roch theory for 
artinian monomial ideals.
\end{abstract}

\section{Introduction}

We examine the Riemann-Roch theorem on a finite graph $G$, due to Baker and Norine \cite{BakNor07},
through the lens of combinatorial commutative algebra. 
Throughout this paper, $G$ is undirected and connected, has $n$ nodes, and multiple edges are allowed, but we do not allow loops.
Its Laplacian is a symmetric $n \times n$-matrix $\Lambda_G$ with
non-positive integer entries off the diagonal and kernel
spanned by ${\bf e} = (1,1,\ldots,1)$. Divisors on $G$ are identified with Laurent monomials ${\bf x^u} = x_1^{u_1} x_2^{u_2} \cdots x_n^{u_n}$.
The {\em chip firing} moves are binomials ${\bf x^u - x^v}$ where 
${\bf u},{\bf v} \geq 0$ and ${\bf u-v}$ is in the lattice spanned by the columns of $\Lambda_G$.
The lattice ideal $I_G$ spanned by such binomials
is here called the {\em toppling ideal} of the graph $G$.
It was  introduced by Perkinson, Perlman and Wilmes \cite{Perkinson, Wil10},
following an earlier study of the inhomogeneous version of $I_G$ by 
Cori, Rossin and Salvy \cite{CorRosSal02}.

For any fixed node, the toppling ideal $I_G$ has a 
distinguished initial monomial ideal $M_{G}$. This
monomial ideal was studied by Postnikov and Shapiro \cite{PosSha}, and it is
characterized by the property that the standard monomials of $M_G$
are the {\em $G$-parking functions}. We construct 
free resolutions for both $I_G$ and $M_{G}$, 
and we study their role for Riemann-Roch theory on~$G$.
For an illustration, consider
 the complete graph on four nodes, $G = K_4$.
The chip firing moves  on $K_4$ are the integer
linear combinations of the columns~of
\begin{equation}
\label{eq:kayfour1} \Lambda_{G} \,\, = \,\, \begin{bmatrix}
\phantom{-} 3 & - 1 & - 1 & - 1\, \,\\
- 1 & \phantom{-} 3  & - 1 & - 1 \,\\
- 1 & - 1 & \phantom{-} 3 & -1 \,\\
- 1 & - 1 & - 1 & \phantom{-} 3 \,\end{bmatrix} .
\end{equation}
The toppling ideal is the lattice ideal in $\mathbb{K}[{\bf x}] = \mathbb{K}[x_1,x_2,x_3,x_4]$ that represents
$\,{\rm image}_\mathbb{Z}(\Lambda_G)$:
\begin{equation}
\label{eq:kayfour2}
 \begin{matrix} I_G \,\,= &
\bigl\langle \,
\underline{x_1^3} - x_2 x_3 x_4,\,
\underline{x_2^3} - x_1 x_3 x_4,\,
\underline{x_3^3} - x_1 x_2 x_4,\,
\underline{x_1 x_2 x_3} - x_4^3 , \\ & 
\underline{x_1^2 x_2^2} - x_3^2 x_4^2 \,,\,\,
\underline{x_1^2 x_3^2} - x_2^2 x_4^2 \,,\,\,
\underline{x_2^2 x_3^2} - x_1^2 x_4^2\,
\bigr\rangle .
\end{matrix}
\end{equation}
This ideal is {\em generic} in the sense of
Peeva and Sturmfels \cite{PeevaStr98}, as
each of the seven binomials contains all four variables.
The minimal free resolution is given by the {\em Scarf~complex}
\begin{equation}
\label{eq:kayfour3}
 0 \,\longleftarrow\, \mathbb{K}[{\bf x}]
 \,\longleftarrow \, \mathbb{K}[{\bf x}]^7  \,\longleftarrow \, 
 \mathbb{K}[{\bf x}]^{12}  \,\longleftarrow \, 
 \mathbb{K}[{\bf x}]^6
  \,\longleftarrow \, 0.
\end{equation}
The seven binomials in (\ref{eq:kayfour2})
form a Gr\"obner basis of $I_G$,
with the underlined monomials generating
the initial ideal $M_G$. That monomial ideal
 has the irreducible decomposition
$$ M_G \,\,\, = \,\,\,
\langle x_1, x_2^2, x_3^3 \rangle \,\cap \,
\langle x_1, x_2^3, x_3^2 \rangle \,\cap \,
\langle x_1^2, x_2, x_3^3 \rangle \,\cap \,
\langle x_1^2, x_2^3, x_3 \rangle \,\cap \,
\langle x_1^3, x_2, x_3^2 \rangle \,\cap \,
\langle x_1^3, x_2^2, x_3 \rangle .
$$
The ideal $M_G$ is the {\em tree ideal} of
\cite[\S 4.3.4]{MilStr05}. Its standard monomials
are in bijection with the $16$ spanning trees.
Its Alexander dual is generated by the six socle elements 
\begin{equation}
\label{eq:kayfour4}
x_2 x_3^2,\,  x_2^2 x_3, \,x_1 x_3^2,\, x_1 x_2^2,\, x_1^2 x_3, \,x_1^2 x_2 .
\end{equation}
These correspond to the {\em maximal
parking functions} studied in combinatorics; see~\cite{BCT, PosSha}.
We claim that the duality seen in Figures 4.2 and 4.3 of~\cite{MilStr05} 
is the same as that expressed in the Riemann-Roch Theorem for $G$.
This  will be made precise in Sections 3 and 4.

The present article is organized as follows:
Section 2 is concerned with the case when $G$ is
a {\em saturated graph}, meaning that any two
nodes $i$ and $j$ are connected by at least one edge.
We show that here $I_G$ is a generic lattice ideal, and
 we determine its minimal free resolution and its
Hilbert series in the finest grading.
The Scarf complex of the initial monomial ideal $M_G$ is supported on the
barycentric subdivision of the $(n-2)$-simplex \cite[\S 6]{PosSha}, and this lifts to the 
Scarf complex of the lattice ideal $I_G$ by \cite[Corollary 5.5]{PeevaStr98}.

In Section 3 we revisit the Riemann-Roch formula
\begin{equation}
\label{eq:RR}
 {\rm rank}(D) \,-\, {\rm rank}(K {-} D) \,\, =\,\, {\rm degree}(D) - {\rm genus} + 1.
 \end{equation}
 We prove this formula in an entirely new setting:
the role of the curve is played by a monomial ideal, and that of the divisors $D$ and $K$
is played by monomials ${\bf x^b}$ and ${\bf x^K}$.
 The identity (\ref{eq:RR}) is shown for  monomial ideals that are
artinian, level, and reflection-invariant. This includes the parking function ideals
 $M_G$ derived from saturated graphs~$G$.
 
 In Section 4 we extend our results to the case
 of graphs $G$ that are not saturated, and we rederive
Riemann-Roch for graphs  as a corollary.
   Here $M_G$ is still an initial ideal of $I_G$, but the choice of term order is more
 delicate \cite[\S 5]{Perkinson}. One choice is the cost function used
 by Baker and Shokrieh for the integer program in \cite[Theorem 4.1]{BakSho11}.
The Scarf complexes in Section 2 support
cellular free resolutions of $I_G$ and $M_G$, but these resolutions are
usually far from minimal. We conclude with
 several open questions.

This paper demonstrates how Riemann-Roch theory embeds into
combinatorial commutative algebra. Our main results are
Theorems \ref{satminres_theo}, \ref{RR_theo} and \ref{nonsatinit_theo}.
These build on earlier works, notably \cite{AmiMan10}
and \cite{PosSha}, but they go much further and are new in their current form.

When this collaboration started in the summer of 2011,
both authors  were unaware of the articles \cite{Perkinson, Wil10} written 
on similar topics by David Perkinson and his students at Reed College.
As our point of departure,
we chose to focus on chip firing in the most  classical case of undirected graphs, 
but with the tacit understanding that our ideals and modules generalize to directed graphs,
arithmetic graphs, simplicial complexes, matroids, abelian networks, or
any of the other extensions seen in the recent chip firing literature (cf.~\cite{AB}).

\section{Saturated graphs}\label{satgraph_sect}

In  this section, we assume that the graph $G$ has $u_{ij}$ edges between node $i$ and node $j$, where $u_{ij}$ 
is a positive integer, for $i \not= j$. However, we do not allow loops, so that $u_{11} = u_{22} = \cdots = u_{nn}   = 0$.
 Thus, in the language of \cite{PosSha},  $G$ is a {\em saturated graph}. We shall see that, under this hypothesis,  the lattice ideal $I_G$ is generic, and an explicit combinatorial description of its minimal free resolution can be given. Throughout this paper we work in the polynomial ring $\mathbb{K}[{\bf x}] = \mathbb{K}[x_1,\ldots,x_n]$ over an arbitrary field $\mathbb{K}$. 

We begin by explicitly showing the generators of the lattice ideal $I_G$ in the case $n=4$.

\begin{example}
\label{ex:fourex}
 \rm
 If $G$ is a saturated graph on $[4] = \{1,2,3,4\}$ then $I_G$ is generated by
$$ \underline{x_1^{u_{12}+u_{13}+u_{14}}} \,- \,x_2^{u_{12}} x_3^{u_{13}} x_4^{u_{14}}\,,\,\,
\underline{x_2^{u_{12}+u_{23}+u_{24}}} \,-\, x_1^{u_{12}} x_3^{u_{23}} x_4^{u_{24}}\,,\,\,
\underline{x_3^{u_{13}+u_{23}+u_{34}}} \,-\, x_1^{u_{13}} x_2^{u_{23}} x_4^{u_{34}}, $$
$$ \underline{x_1^{u_{14}} x_2^{u_{24}} x_3^{u_{34}}} \,\,-\,\,  x_4^{u_{14}+u_{24}+u_{34}}  \,,\,\,
\underline{x_1^{u_{13}+u_{14}} x_2^{u_{23}+u_{24}}} \,\,- \,\, x_3^{u_{13}+u_{23}} x_4^{u_{14}+u_{24}}\,, $$
$$ \underline{x_1^{u_{12}+u_{14}} x_3^{u_{23}+u_{34}}} \,- x_2^{u_{12}+u_{23}} x_4^{u_{14}+u_{34}}\,,\,\,
\underline{x_2^{u_{12} +u_{24}} x_3^{u_{13}+u_{34}}} \,-\,
x_1^{u_{12}+u_{13}} x_4^{u_{24}+u_{34}}.
$$
Here the $u_{ij}$ are arbitrary positive integers. These binomials form a Gr\"obner basis.
The initial ideal $M_G$ is generated by the underlined monomials.
The minimal free resolution of $I_G$ has the form (\ref{eq:kayfour3}).
The same holds for $M_G$, as was shown in \cite[Corollary 6.9]{PosSha}.
The minimal  resolution of $M_G$ is given by the Scarf complex, which is depicted
in Figure~\ref{fig:eins}.
 \qed
\end{example}

We now state our main result in this section.
For disjoint subsets $I$ and $J$ of $[n]$ we~set
$$ {\bf x}^{I \rightarrow J} \,\,\, := \,\,\,
 \prod_{i \in I} x_i^{\sum_{k \in J} u_{ik}}. $$
 A {\em split} of the set $[n] = \{1,2,\ldots,n\}$ is an 
unordered pair $(I,J)$ of non-empty disjoint subsets $I$ and $J$ 
whose union equals $[n]$. 
The number of splits equals $2^{n-1}-1$.
With each split $(I,J)$ we associate
the following binomial which is well-defined up to sign:
\begin{equation}
\label{eq:IJbinomial}
{\bf x}^{I \rightarrow J} \, -\, {\bf x}^{J \rightarrow I}.
     \end{equation}
These are precisely the seven binomials in 
Example \ref{ex:fourex}, one for each split $(I,J)$.

Let ${\rm Cyc}_{n,k}$ denote the set of
cyclically ordered partitions of the set $[n]$
into $k$ blocks. Each element of ${\rm Cyc}_{n,k}$
has the form $(I_1 , I_2 , \ldots ,I_k )$, where
$I_1 \cup I_2 \cup \cdots \cup I_k = [n]$ is a partition.
We regard the $(I_1,I_2,\ldots,I_k)$ as formal symbols,
subject to the identifications
$$ 
(I_1 , I_2 , \ldots , I_{k-1} ,I_k ) \,\, =\,\,
( I_2 , I_3 , \ldots ,I_k , I_1 ) \,\, = \,\,\, \cdots \,\,\, = \,\,
(I_k , I_1 , \ldots , I_{k-2} , I_{k-1} ).
$$
We write $\mathbb{K}[{\bf x}]^{{\rm Cyc}_{n,k}}$ for the free
$\mathbb{K}[{\bf x}]$-module generated by these symbols.
The rank of this free module equals
the number of cyclically ordered partitions, namely
\begin{equation}
\label{eq:cycnumber}
 |{\rm Cyc}_{n,k}| \,\,\,= \,\,\, (k-1)! \cdot S_{n,k} , 
 \end{equation}
where $S_{n,k}$ is the Stirling number of the second kind, i.e.,
the number of partitions of the set $[n]$ into $k$ blocks.
Let $\,\mathcal{CYC}_G\,$ denote the following  complex of free $\mathbb{K}[{\bf x}]$-modules:
\begin{equation}
\label{eq:cycres1}
 0 \,\longleftarrow \, \mathbb{K}[{\bf x}]^{{\rm Cyc}_{n,1}}
          \,\longleftarrow\,\mathbb{K}[{\bf x}]^{{\rm Cyc}_{n,2}}
                    \,\longleftarrow\,\mathbb{K}[{\bf x}]^{{\rm Cyc}_{n,3}} \,\longleftarrow \,\,\cdots \,
          \,\longleftarrow\,\mathbb{K}[{\bf x}]^{{\rm Cyc}_{n,n}}   \longleftarrow \, 0,          
 \end{equation}
          where the boundary map from $\,  \mathbb{K}[{\bf x}]^{{\rm Cyc}_{n,r}}\,$
          to   $\,  \mathbb{K}[{\bf x}]^{{\rm Cyc}_{n,r-1}}\,$ is given by the formula
\begin{equation}
\label{eq:cycres2}
 (I_1, I_2, I_3 , \ldots ,I_r) \,\,\,  \mapsto \quad
\begin{matrix} \sum_{s=1}^{r-1} (-1)^{s-1}
{\bf x}^{I_s \rightarrow I_{s+1}}  
(I_1, \ldots ,I_{s-1}, I_s \cup I_{s+1} ,I_{s+2}, \ldots , I_r)  \\
 - \,\,{\bf x}^{I_r \rightarrow I_1} \cdot (I_2,I_3, \ldots, I_{r-1}, I_1 \cup I_r).  \qquad \qquad \qquad \qquad 
 \end{matrix} .
\end{equation}
In this formula it is assumed that $n \in I_r$, so as to ensure that all signs are consistent.

\begin{figure}
 \includegraphics[width=15.8cm]{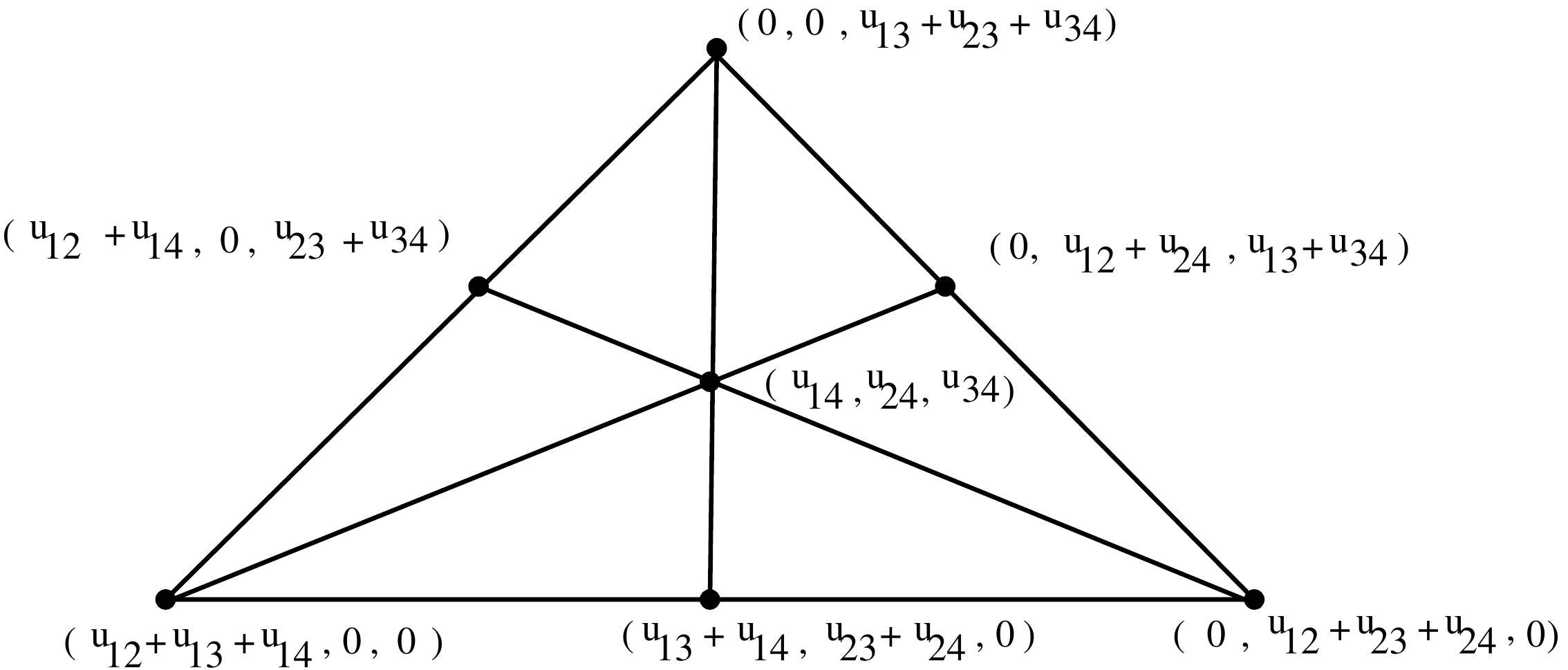} 
\caption{The barycentric subdivision of the $(n-2)$-simplex
supports minimal free resolutions for 
toppling ideals of saturated graphs with $n$ nodes and their initial ideals.
Shown here is the Scarf complex (\ref{eq:kayfour3}) for $n=4$,
with $7$ vertices, $12$ edges and $6$ triangles.
}
\label{fig:eins}
\end{figure}

\begin{theorem}\label{satminres_theo}
Let $G$ be a saturated graph.
The toppling ideal $I_G$ is a generic lattice ideal. 
It is minimally generated by the 
$2^{n-1}{-}1$ binomials  (\ref{eq:IJbinomial}), these
form a reverse lexicographic Gr\"obner basis, the 
complex $\,\mathcal{CYC}_G$ coincides with the Scarf complex,
and this complex minimally resolves~$ \mathbb{K}[{\bf x}]/I_G$.
\end{theorem}

\begin{proof}
We begin by noting that 
$\,{\bf x}^{I \rightarrow J}  - {\bf x}^{J \rightarrow I} \,$
actually lies in the ideal $I_G$. To see this, let
$e_I$ denote the incidence vector in $\{0,1\}^n$ that
represents the subset $I$ of $[n]$.
The $i$-th coordinate of the vector $\Lambda_G \cdot e_I$ is equal to
$\, \sum_{k \in J} u_{ik} $ if $i \in I$, and it is
$\,- \sum_{k \in I} u_{ik}$ if $i \in J$. Hence
$\Lambda_G \cdot e_I$ is represented algebraically by 
$\,{\bf x}^{I \rightarrow J}  - {\bf x}^{J \rightarrow I} $, which is hence in $I_G$.

Fix any reverse lexicographic term order on $\mathbb{K}[{\bf x}]$ 
that has $x_n$ as the smallest variable,
and let ${\rm in}(I_G)$ denote
the initial monomial ideal of $I_G$. Since $I_G$ is a lattice ideal,
 $x_n$ is a non-zerodivisor and it does not divide any of the
 generators of  ${\rm in}(I_G)$. We may thus regard
 ${\rm in}(I_G)$ as an artinian ideal in
  $\mathbb{K}[{\bf x}_{\backslash n}] = \mathbb{K}[x_1,\ldots,x_{n-1}]$.
The index of the Laplacian lattice ${\rm image}_{\mathbb{Z}}(\Lambda_G) $ in
its saturation $\{u \in \mathbb{Z}^n: u_1 + \cdots + u_n = 0\}$ equals the number $T_G$
of spanning trees of $G$. Hence ${\rm in}(I_G)$ has
$T_G$ standard monomials in $\mathbb{K}[{\bf x}_{\backslash n}] $.

Let $M_G$ denote the ideal generated by the  initial monomials of the binomials
in (\ref{eq:IJbinomial}):
\begin{equation}
\label{eq:MG}
 M_G \,\,\, = \,\,\, 
\bigl\langle \,{\bf x}^{I \rightarrow [n]\backslash I} \,:\,
\hbox{$I$ non-empty subset of $[n-1]$} \,\bigr\rangle. 
\end{equation}
By construction, the inclusion $M_G \subseteq {\rm in}(I_G)$ holds.
The monomial ideal $M_G$ was studied in \cite{PosSha}
and shown to have precisely $T_G$ standard monomials.
Indeed,  the standard monomials of $M_G$ are in bijection with the
{\em $n$-reduced divisors}. It is known in the chip firing literature
(cf.~\cite{BakNor07, BCT, CorRosSal02}) that their number equals the number $T_G$ of spanning trees.
Hence $M_G$ and ${\rm in}(I_G)$ are artinian of the same colength in 
$\mathbb{K}[{\bf x}_{\backslash n}] $, so they must be equal:
$$ M_G \,\,\, = \,\,\, {\rm in}(I_G). $$
Therefore the binomials (\ref{eq:IJbinomial})
form a Gr\"obner basis, and hence a generating set, of $I_G$.

The ideal $I_G$ is a generic lattice ideal, in the sense of 
\cite{PeevaStr98}, because all $n$ variables $x_1,\ldots,x_n$
occur in the binomial (\ref{eq:IJbinomial}). Here we are using
that $G$ is saturated. By \cite[Theorem 4.2]{PeevaStr98},
the Scarf complex is the (essentially unique) minimal free
resolution of~$I_G$.

It remains to be seen that the Scarf complex is equal to $\mathcal{CYC}_G$.
Postnikov and Shapiro \cite[Corollary 6.9]{PosSha} showed that
the Scarf complex of the initial ideal $M_G$ is supported on the 
barycentric subdivision of the $(n-2)$-simplex, as shown in Figure~\ref{fig:eins}. 
The Scarf resolution has the format
(\ref{eq:cycres1}), but with $\mathbb{K}[{\bf x}]$ replaced by $\mathbb{K}[{\bf x}_{\backslash n}]$.
Here,  we label the cells in that barycentric subdivision  with ordered partitions
$(I_1,I_2,\ldots,I_r)$ satisfying $n \in I_r$.
The boundary maps in the Scarf resolution are then given by
(\ref{eq:cycres2}), namely, by the sum ranging from $s=1$ to $s=r-1$,
but without the additional term
$\, - \,{\bf x}^{I_r \rightarrow I_1} \cdot (I_2,I_3, \ldots, I_{r-1}, I_1 \cup I_r)$.

We pass from the Scarf resolution of $M_G$ to that of
$I_G$ by the combinatorial rule in
\cite[Theorem 5.4]{PeevaStr98}. This adds precisely one
term to the boundary of each Scarf simplex of $M_G$.
In our case, that additional term is precisely the one above,
and we get (\ref{eq:cycres2}).
\end{proof}

\begin{example} \label{ex:ressi}
\rm
Returning to Example \ref{ex:fourex},
with the seven binomials in that order,
here are the matrices over $\mathbb{K}[x_1,x_2,x_3,x_4]$ that represent
 the first and second syzygies in $\mathcal{CYC}_4$:
$$ 
\bordermatrix{
  &  {}_{(1,2,34)} & {}_{(2,1,34)} & {}_{(1,3,24)} & {}_{(3,1,24)} & {}_{(2,3,14)} & 
  {}_{(3,2,14)} & {}_{(1,23,4)} & {}_{\cdots}    & {}_{(12,3,4)}
    \cr
{}_{(1,234)} & \! - {\bf x}^{2 \rightarrow 34} &\!\! - {\bf x}^{34 \rightarrow 2} & \! \! - {\bf x}^{3 \rightarrow 24} & 
\! - {\bf x}^{24 \rightarrow 3} & 0 & 0 & \!\! - {\bf x}^{23 \rightarrow 4} & \cdots & 0 
\cr 
{}_{(2,134)} & \! - {\bf x}^{34 \rightarrow 1} & \!\! -{\bf x}^{1 \rightarrow 34} & 0 & 0 & 
\!\! -{\bf x}^{3 \rightarrow 14} & \! \! - {\bf x}^{14 \rightarrow 3} & 0 & \cdots & 0 
\cr
{}_{(3,124)} &  0 & 0 & \! - {\bf x}^{24 \rightarrow 1} & \! \!- {\bf x}^{1 \rightarrow 24} & \! - {\bf x}^{14 \rightarrow 2} & 
\!\!  - {\bf x}^{2 \rightarrow 14} & 0 & \cdots & \! - {\bf x}^{4 \rightarrow 12}
\cr
{}_{(123,4)} & 0 & 0 &  0 & 0 & 0 & 0 & {\bf x}^{1 \rightarrow 23} & \cdots & {\bf x}^{12 \rightarrow 3}
\cr
{}_{(12,34)} & {\bf x}^{1 \rightarrow 2} & {\bf x}^{2 \rightarrow 1} & 0 & 0 & 0 & 0 & 0 & \cdots &
\! - {\bf x}^{3 \rightarrow 4}
\cr
{}_{(13,24)} & 0 & 0 & {\bf x}^{1 \rightarrow 3} & {\bf x}^{3 \rightarrow 1} & 0 & 0 & 0 & \cdots & 0 
\cr
{}_{(23,14)} & 0 & 0 & 0 & 0 & {\bf x}^{2 \rightarrow 3} & {\bf x}^{3 \rightarrow 2} & \!-{\bf x}^{4 \rightarrow 1} & 
\cdots & 0
\cr
}.
$$
$$
\bordermatrix{& 
 {}_{(1,2,3,4)} &
 {}_{(1,3,2,4)} &
 {}_{(2,1,3,4)} &
 {}_{(2,3,1,4)} &
 {}_{(3,1,2,4)} &
 {}_{(3,2,1,4)}  \cr
 {}_{(1,2,34)} &  {\bf x}^{3 \rightarrow 4} & 0 & 0 & 0 &\!- {\bf x}^{4 \rightarrow 3} & 0 \cr
 {}_{(2,1,34)} &   0 & 0 &  {\bf x}^{3 \rightarrow 4}  & 0 & 0 & \!-{\bf x}^{4 \rightarrow 3} \cr
 {}_{(1,3,24)} &   0 & {\bf x}^{2 \rightarrow 4} & \! -{\bf x}^{4 \rightarrow 2} &  0 & 0 & 0 \cr
 {}_{(3,1,24)} &   0 & 0 & 0 & \! - {\bf x}^{4 \rightarrow 2} & {\bf x}^{2 \rightarrow 4} & 0 \cr
 {}_{(2,3,14)} &  \!-{\bf x}^{4 \rightarrow 1} & 0 & 0 & {\bf x}^{1 \rightarrow 4} & 0 & 0 \cr
 {}_{(3,2,14)} &   0 & \! - {\bf x}^{4 \rightarrow 1} & 0 & 0 & 0 & {\bf x}^{1 \rightarrow 4} \cr
 {}_{(1,23,4)} &  \!-{\bf x}^{2 \rightarrow 3} & \! - {\bf x}^{3 \rightarrow 2} & 0 & 0 & 0 & 0 \cr
 {}_{(23,1,4)} &  0 & 0 & 0 & {\bf x}^{2 \rightarrow 3} & 0 & {\bf x}^{3 \rightarrow 2} \cr
 {}_{(2,13,4)} &  0 & 0 & \! - {\bf x}^{1 \rightarrow 3}  & \! - {\bf x}^{3 \rightarrow 1} & 0 & 0 \cr
 {}_{(13,2,4)} &   0 & {\bf x}^{1 \rightarrow 3} & 0 & 0 & {\bf x}^{3 \rightarrow 1} & 0\cr
{}_{(3,12,4)} &  0 & 0 & 0 & 0 & \! - {\bf x}^{1 \rightarrow 2} & \!-{\bf x}^{2 \rightarrow 1} \cr
 {}_{(12,3,4)} &  {\bf x}^{1 \rightarrow 2} & 0 & {\bf x}^{2 \rightarrow 1}&  0 & 0 & 0\cr
 }.
 $$
 Note that the seven binomial generators of $I_G$ appear as the $2 \times 2$-minors of the
 $3 \times 2$-matrices seen in the six pairs of columns within the
   $7 \times 12$-matrix of first syzygies.    The syzygies of the ideal $M_G$
   generated by the underlined monomials in Example \ref{ex:fourex}
      are found by replacing with $0$ all monomials  that have the symbol ``4'' to the left of the arrow.    \qed
 \end{example}

 One immediate application of our minimal free resolution is a formula 
 for the Hilbert series of the ring $\mathbb{K}[{\bf x}]/I_G$ in its natural
 grading by the  group ${\rm Div}(G) = \mathbb{Z}^n/{\rm image}_\mathbb{Z}(\Lambda_G)$.
As is customary in chip firing theory \cite{AB, BakNor07, BakSho11, Perkinson}, we consider the 
decomposition
$$ {\rm Div}(G) \,\, = \,\, \mathbb{Z} \,\oplus \,{\rm Div}_0(G), $$
where $\mathbb{Z}$ records the degree of a divisor on $G$, and
${\rm Div}_0(G)$ is the finite subgroup of divisors of degree $0$.
The order of ${\rm Div}_0(G)$  is the number of spanning trees of $G$.
Let $t$ and $q$ denote the generators of the group algebra
$\mathbb{Z}[{\rm Div}(G)]$ corresponding to this decomposition.
The Hilbert series of $\mathbb{K}[{\bf x}]/I_G$  equals $1/(1-t) $ times the Hilbert series of 
$\mathbb{K}[{\bf x}_{\backslash n}]/M_G$, where
$M_G = {\rm in}(I_G)$ is the initial ideal in (\ref{eq:MG}).
The latter series equals 
$$  \sum_{{\bf u}} t^{|{\bf u}|} q^{{\rm div}({\bf u})} .$$
This finite sum is over all elements ${\bf u} \in \mathbb{N}^{n-1}$
that represent parking functions on $G$ with respect to the
last node $n$, and ${\rm div}({\bf u})$ denotes the  class of the
reduced divisor of degree $0$ given by the vector $({\bf u},- \sum u_i)$.
See \cite[Theorem 6.14]{Perkinson} for a nice formula, due to Merino \cite{Merino},
which expresses this sum with $q=1$ in terms of
the Tutte polynomial of $G$.

We fix the natural epimorphism from
the semigroup algebra of $\mathbb{N}^{n-1}$ to that of ${\rm Div}(G)$:
$$ \psi:  \mathbb{Z}[{\bf x}_{\backslash n}] \rightarrow
\mathbb{Z}[{\rm Div}(G)] \,,\,\, \,  {\bf x}^{\bf u} \mapsto t^{|{\bf u}|} q^{{\rm div}({\bf u})} .$$
With this notation, our minimal free resolution
in Theorem \ref{satminres_theo} implies the following result:

\begin{corollary} \label{cor:HS}
The Hilbert series of $K[{\bf x}]/I_G$ in the grading by the group ${\rm Div}(G)$ equals
$$
  \frac{1}{1-t} \sum_{u} t^{|u|} q^{{\rm div}(u)}  \,\,\, = \,\,\,
   \frac{
1- \sum_{k=1}^n (-1)^k \sum_{(I_1,I_2,\ldots,I_k) \in {\rm Cyc}_{n,k}}
\psi({\bf x}^{I_1 \rightarrow I_2}{\bf x}^{I_2 \rightarrow I_3} \cdots{\bf x}^{I_{k-1} \rightarrow I_k})}{
(1-t)(1-\psi(x_1))(1-\psi(x_2))\, \cdots \,(1-\psi(x_{n-1}))} .$$
\end{corollary}

\begin{proof}
It suffices to note that the $\mathbb{Z}^{n-1}$-degree of the
basis element $(I_1,I_2,\ldots,I_k)$ of the free
$\mathbb{K}[{\bf x}_{\backslash n}]$-module in the
$k$-th step of the resolution of $M_G$ is the
exponent vector of
${\bf x}^{I_1 \rightarrow I_2}{\bf x}^{I_2 \rightarrow I_3} \cdots{\bf x}^{I_{k-1} \rightarrow I_k}$.
This monomial does not contain $x_n$ because $n \in I_k$. 
\end{proof}

We close this section with a combinatorial recipe for the
socle monomials modulo $M_G$. These are
the monomials ${\bf x}^{\bf u}$ that are not in $M_G$ but
${\bf x}^{\bf u} x_i \in M_G$ for $i=1,\ldots,n-1$.
Each permutation of $[n-1]$ corresponds to a 
flag $\mathcal{T}$ of subsets
$\emptyset \subset T_1 \subset T_2 \subset T_3 \subset \dots \subset T_{n-1}$.
The flag is complete, meaning that each inclusion is strict
and each $T_i \backslash T_{i-1}$ is a singleton.
Let ${\overline T_i}$ denote the set complement of $T_i$ with respect to $[n]$.
For instance, ${\overline T_{n-1}} = \{n\}$.

\begin{corollary}\label{gensocle_theo} The socle monomials of $\mathbb{K}[{\bf x}_{\backslash n}]/M_G$ are precisely 
the $(n-1)!$ monomials
\begin{equation} \label{eq:MGsoc}
{\bf s}_{\mathcal{T}}\,\,= \,\,{\rm lcm}({\bf x}^{T_1 \rightarrow {\overline T_1}},{\bf x}^{T_2 \rightarrow {\overline T_2}},\dots,{\bf x}^{T_{n-1} \rightarrow {\overline T_{n-1}}})/(x_1 x_2 \cdots  x_{n-1}),
\end{equation}
where ${\mathcal{T}}$ runs over all complete flags of subsets  of $[n-1]$.
\end{corollary}

\begin{proof}
The Scarf complex of $M_G$ is a minimal free resolution and it is
 supported on  the barycentric 
subdivision of the $(n-2)$-simplex,
by  \cite[Corollary 6.9]{PosSha} and our discussion above.
Each facet in that barycentric subdivision corresponds to a 
complete flag $\mathcal{T}$. 
The vertices of that facet are labeled by
${\bf x}^{T_1 \rightarrow {\overline T_1}},{\bf x}^{T_2 \rightarrow {\overline T_2}},
\ldots, {\bf x}^{T_{n-1} \rightarrow {\overline T_{n-1}}}$ in the Scarf complex,
and the monomial label of the facet is their least common multiple.
Facets of the Scarf complex are in  bijection with the
irreducible components of $M_G$ and also with the socle monomials modulo $M_G$.
By \cite[Corollary 6.20]{MilStr05}, each socle monomial is multiplied by the
product of all variables to give the monomial label of the corresponding facet.
\end{proof}

\begin{remark} \rm
Our results hold verbatim for all generic sublattices of
finite index in the root lattice $\,A_n = \{u \in \mathbb{Z}^n: \sum_{i=1}^n u_i = 0\}$,
so we recover the Voronoi theory of~\cite{AmiMan10, AB}.
We posit that our commutative algebra
derivation of their Voronoi theory
is a natural and useful one, and that it opens up
new and unexpected connections.
For instance, Gr\"obner bases of lattice ideals are fundamental for
integer programming \cite{Tho98}.
One original source for that application is Herbert Scarf's seminal work on
neighborhood systems in economics. A key example that motivated Scarf
was the {\em Leontief system} \cite[\S 2A]{Sca86}. It turns out
that the lattices representing Leontief systems
are precisely our generic lattices here.
The Gr\"obner basis property stated in Theorem \ref{satminres_theo}
is in fact equivalent to~\cite[Theorem 2.2]{Sca86}. \qed
\end{remark}

\section{A Riemann-Roch Theorem for Monomial Ideals}

In this section we fix an arbitrary artinian monomial ideal
$M$ in a polynomial ring $\mathbb{K}[\mathbf{x}] = \mathbb{K}[x_1,\ldots,x_m]$.
 We focus on Alexander duality \cite[\S 5]{MilStr05}, and we establish the Riemann-Roch 
formula (\ref{eq:RR}) in this new context.  Towards the end of this section, and in the next section, we will  
recover the Riemann-Roch formula for graphs from the Riemann-Roch formula for monomial ideals. To begin with, we need to gather the ingredients, that is, we need to 
redefine the notions of divisor, genus,  rank and degree. 

The role of {\em divisors} on the monomial ideal $M$ is played by Laurent monomials ${\bf x^b}$.

\begin{definition}{\rm ({\bf Rank of a monomial})} \rm \label{def:rank}
The {\em rank} of a monomial ${\bf x^b}$ 
with respect to $M$ is one less than the minimum degree of any monomial ${\bf x^a}$  that satisfies
$\,{\bf x^b} \in \langle {\bf x^a} \rangle \backslash {\bf x^a } M$.
\end{definition}

This definition is restricted to honest monomials ${\bf x^b}$, where ${\bf b} \geq 0$.
Just before the statement of Theorem \ref{RR_theo}, we shall extend the definition of rank to all Laurent monomials.

The rank measures how deeply a monomial ${\bf x^b}$ sits inside the ideal $M$.
We have ${\rm rank}({\bf x^b}) \geq 0$ if ${\bf x^b} \in M$ and
${\rm rank}({\bf x^{b}}) = -1$ otherwise. 
Rank zero monomials form the {\em border} of $M$.
Let $\,{\rm MonSoc}(M) \,= \, \bigl\{{\bf x^c} \not\in M \,|\,\,x_i {\bf x^c} \in M
\,\, \forall i  \,\bigr\}$ denote the set of {\em socle monomials} of $\mathbb{K}[{\bf x}]/M$. 
See Figure \ref{fig:zwei} for a picture of a monomial ideal.
The ideal generators are the large black circles, and monomials in $M$
are labeled by their rank. The socle elements are the black squares,
and  other standard monomials are white squares.

\begin{definition}{\rm ({\bf Reflection invariance})} \rm
A monomial ideal $M$ is {\em reflection-invariant} if there exists a
{\em canonical monomial} ${\bf x^{K}}$ such that the
 map $\,\phi \,: \, \bf x^{c} \mapsto {\bf x^{K}}/{\bf x^{c}}$ defines
 an involution of the set  $ {\rm MonSoc}(M)$.
 This requires that every socle monomial divides ${\bf x^K}$.
 \end{definition}
 
Using notation as in \cite[\S 5]{MilStr05}, we note that
our artinian monomial ideal $M$ is reflection-invariant
with canonical monomial ${\bf x^K}$ if and only if the following identity holds:
\begin{equation}
\label{eq:reflinv}
 M^{[{\bf K}+ {\bf e}]} \,\,\, = \,\,\, \bigl\langle\, {\rm MonSoc}(M) \,\bigr\rangle .
\end{equation}
Here ${\bf e} = (1,1,\ldots,1)$ and 
$M^{[{\bf K}+ {\bf e}]}$ is the {\em Alexander dual} of $M$ with respect to
${\bf K}+{\bf e}$.

\begin{definition}{\rm ({\bf Genus})} \rm
The ideal $M$ is {\em level}
if all socle monomials have the same degree.
If this holds then one plus that degree is called the {\em genus} of $M$, 
denoted $g = {\rm genus}(M)$.
\end{definition}

\begin{example} \rm
Let $M$ be the ideal generated by the seven underlined monomials in (\ref{eq:kayfour2}).
Then $M$ is level of genus $g = 4$, because all six socle monomials in (\ref{eq:kayfour4}) are cubics,
and $M$ is reflection-invariant. The canonical monomial ${\bf x^K} = x_1^2 x_2^2 x_3^2$ has
degree $ 2 g - 2 = 6$. 
But, the rank of ${\bf x^K}$ is equal to $g-2 = 2$, as can be seen from the following lemma. \qed
\end{example}

For ${\bf u} = (u_1,\ldots,u_m) \in {\bf Z}^m$ we abbreviate
  $\,{\rm degree}^{+}({\bf u})=\sum_{i:u_i > 0}u_i$.

\begin{lemma}\label{rankformart_theo}
Let $M$ be an artinian monomial ideal. Then every monomial ${\bf x^b}$ satisfies
\begin{equation} \label{rank_form}
                 {\rm rank}({\bf x^b}) \quad =\,\min_{{\bf x^c} \in {\rm MonSoc}(M)}{\rm degree}^{+}({\bf b-c})-1.
\end{equation}
\end{lemma}

\begin{proof}
The condition $\,{\bf x^b} \in \langle {\bf x^a} \rangle \backslash {\bf x^a } M\,$
in Definition \ref{def:rank} is equivalent to
 $\, {\bf x}^{{\bf b}-{\bf a}} \in \mathbb{K}[{\bf x}]\backslash M$. Maximizing the degree
of ${\bf x^c} = {\bf x}^{{\bf b}-{\bf a}}$ subject to this condition is equivalent to
minimizing the degree of ${\bf x^a}$. But, since $M$ is
artinian, the maximal degree among its finitely many standard monomials
is attained by one of the socle monomials ${\bf x^c}\in {\rm MonSoc}(M)$. \end{proof}

\begin{remark} \rm
Formula (\ref{rank_form}) resembles the formula  in
 \cite[Lemma 2.2]{BakNor07}  for the rank of a divisor on a finite graph.
 We shall exploit this resemblance at the end of this section. \qed
\end{remark}

\begin{figure}
\qquad \qquad \qquad \includegraphics[width=11.0cm]{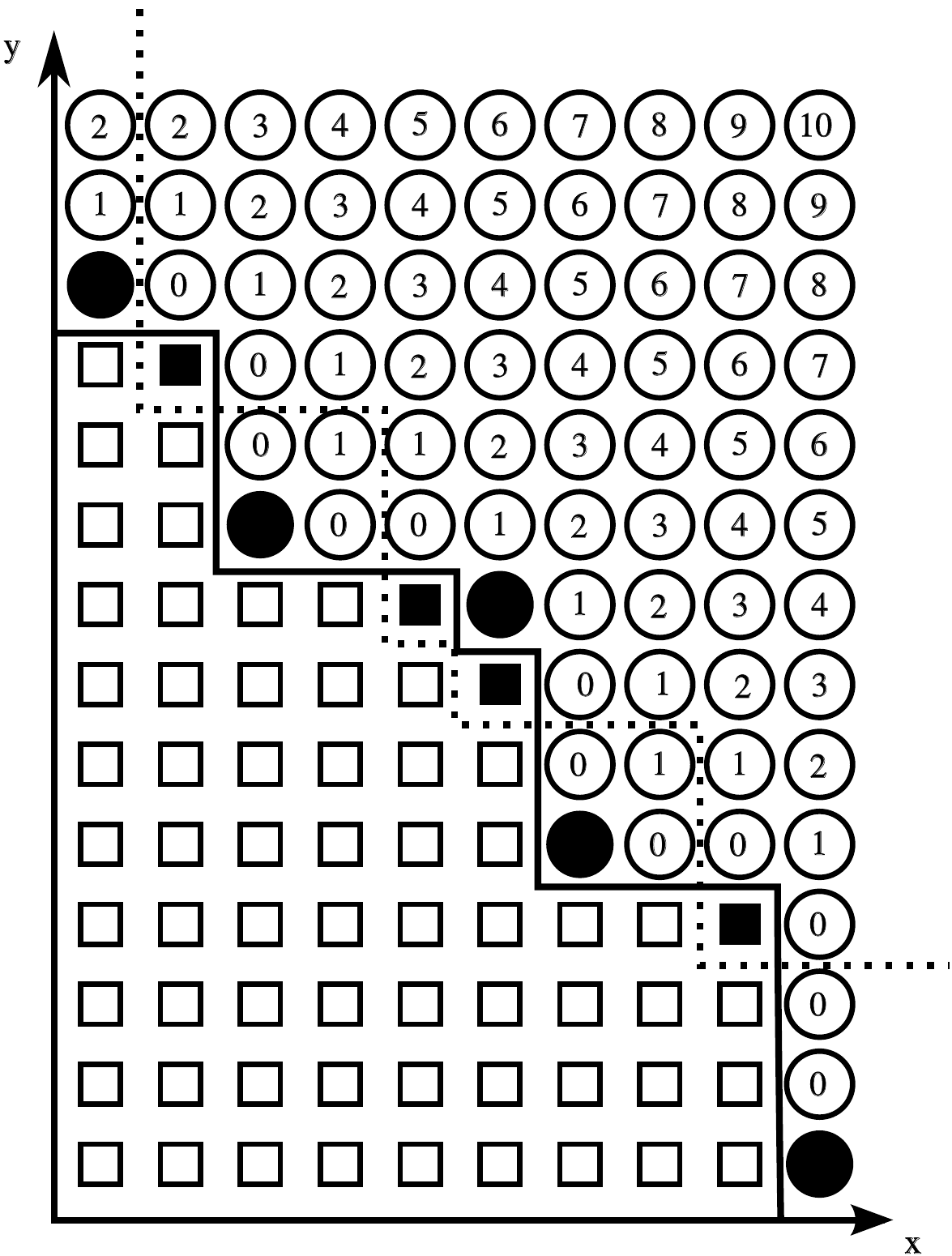} 

\medskip

\caption{The monomial ideal $M = \langle x^9 , x^6 y^4, x^5 y^7, x^2 y^8, y^{11} \rangle$
is Riemann-Roch of genus $12$, with canonical monomial
$x^9 y^{13}$ and ${\rm Soc}(M) = \{x^8 y^3, x^5 y^6, x^4 y^7, x y^{10} \} $.
Generators and socle monomials highlighted in dark.
Monomials in $M$ are labeled with their rank. 
The square boxes correspond to the standard monomials of $\mathbb{K}[{\bf x}]/M$. The dotted lines mark the boundary of the staircase region of $M^{[{\bf K}+ {\bf e}]}$.
Note that the identity (\ref{eq:reflinv}) holds.
}
\label{fig:zwei}
\end{figure}

The formula in (\ref{rank_form}) can be rewritten to be reminiscent of 
 S-pairs for Gr\"obner bases:
\begin{equation}
{\rm rank}({\bf x^b}) \quad =
\,\min_{{\bf x^c} \in \text{MonSoc}(M)}{\rm degree}\left(\frac{{\rm lcm}({\bf x^b},~{\bf x^c}\big)}{{\bf x^c}}\right)-1. 
\end{equation}

We now define the {\em rank} of an arbitrary Laurent monomial ${\bf x^b}$ by the
formula (\ref{rank_form}). This  is consistent with Definition  \ref{def:rank},
and it is the natural extension to monomials some of whose exponents are negative.
The following main result gave this section its title:

\begin{theorem}\label{RR_theo} 
Let $M$ be a monomial ideal that is artinian, level, and reflection-invariant.
Then $M$ satisfies the Riemann-Roch formula, i.e.,~every Laurent monomial ${\bf x^{b}}$ satisfies
\begin{equation} \label{eq:RRF}
{\rm rank}({\bf x^{b}})-{\rm rank}({\bf x^{K}}/{\bf x^{b}})\,\,=\,\,{\rm degree}({\bf x^{b}})- {\rm genus}(M) +1.
\end{equation}
\end{theorem}

\begin{proof}
We denote ${\bf x^{K}}/{\bf x^{c}}$ by ${\bf x^{\bar{c}}}$.
Using the formula for rank shown in Lemma \ref{rankformart_theo}, the left hand side of (\ref{eq:RRF}) equals
\begin{equation}
\label{eq:dreiii}
\min_{{\bf x^c} \in \text{MonSoc}(M)}{\rm degree}^{+}({\bf b-c}) \quad -\, 
 \min_{{\bf x^c} \in \text{MonSoc}(M)}{\rm degree}^{+}({\bf \bar{c}-b}). 
\end{equation}
For any socle monomial ${\bf x^c}$ we have ${\rm degree}^{+}({\bf b-c})-{\rm degree}^{+}({\bf c-b})={\rm degree}({\bf x^b})-{\rm degree}({\bf x^c})={\rm degree}({\bf x^b})-({\rm genus}(M)-1)$,
and hence
\begin{equation}\label{dplus_eq}
{\rm degree}^{+}({\bf b-c})={\rm degree}^{+}({\bf c-b})+{\rm degree}({\bf x^b})-({\rm genus(M)}-1). 
\end{equation}
Taking the minimum of ${\rm degree}^{+}({\bf b-c})$ over ${\bf x^c} \in {\rm MonSoc}(M)$, 
equation (\ref{dplus_eq}) implies
\begin{gather}
\label{eq:vierrr}
\hspace{-7.0cm} \min_{{\bf x^c} \in \text{MonSoc}(M)}{\rm degree}^{+}({\bf b-c})  \\
\notag = \min_{{\bf x^c} \in \text{MonSoc}(M)} \bigl( {\rm degree}^{+}({\bf c-b})+{\rm degree}({\bf x^b})-({\rm genus}(M)-1) \bigr) \\
\hspace{0.5cm}\notag=   \left( \min_{{\bf x^c} \in \text{MonSoc}(M)} {\rm degree}^{+}({\bf c-b})\right) \,+\, {\rm degree}({\bf x^b})-({\rm genus}(M)-1).
\end{gather}
Since the map $\phi$ is an involution, we can replace $ {\bf \bar{c}}$ by ${\bf c}$ in
the second row of (\ref{eq:dreiii}). It then follows from (\ref{eq:vierrr}) that (\ref{eq:dreiii}) 
is equal to $\,{\rm degree}({\bf x^b})-{\rm genus}(M)+1$, as desired.
\end{proof}



\begin{remark} \rm Theorem \ref{RR_theo} can be extended
to monomial ideals $M$ that are artinian and reflection invariant 
but not necessarily level. Such $M$ arise as initial ideals
 from directed regular (indegree = outdegree) graphs.
Following \cite[\S 2.3]{AmiMan10}, we define ${\rm genus}_{\rm min}(M)$ 
as one minus the minimum degree of a socle monomial of $M$,
and ${\rm genus}_{\rm max}(M)$   as  one minus the maximum degree of 
a socle monomial of $M$. Using a technique 
similar to that in the proof of Theorem \ref{RR_theo}, we can derive the following Riemann-Roch inequalities:
\begin{equation}\label{RRE_inq}
{\rm genus}_{\rm min}(M)-1 
\,\leq\, {\rm degree}({\bf x^{b}})-{\rm rank}({\bf x^{b}})+{\rm rank}({\bf x^{K}}/{\bf x^{b}})
\,\leq \, {\rm genus}_{\rm max}(M)-1.
\end{equation}
Of course, the above inequality generalizes the Riemann-Roch formula (\ref{eq:RRF}):  if the ideal $M$ is also level then
 ${\rm genus}_{\rm  max}(M)={\rm genus}_{\rm min}(M)={\rm genus}(M)$, and
the Riemann-Roch formula (\ref{eq:RRF}) immediately 
follows from the inequalities (\ref{RRE_inq}). \qed
\end{remark}

We say that a monomial ideal $M$ is {\em Riemann-Roch}
if it is artinian, level, and reflection invariant. See Figure \ref{fig:zwei} for an example 
in two variables.
  In what follows we assume that $M$ is a
  Riemann-Roch monomial ideal. The next corollaries
are formal consequences of  the Riemann-Roch formula,
as is the case  for algebraic curves and graphs.

\begin{corollary}\label{cano_cor} If ${\bf x^b}$ is a multiple of  ${\bf x^K}$ then
       ${\rm rank}({\bf x^{b}})={\rm degree}({\bf x^{b}})-{\rm genus}(M)$. 
       \end{corollary}

\begin{proof} If ${\bf x^K}$ divides ${\bf x^{b}}$, then ${\rm rank}({\bf x^K/x^{b}})=-1$.
 Plugging this equation into the Riemann-Roch formula gives the assertion.
\end{proof}

Note that, by definition,
the degree of the canonical monomial ${\bf x^K}$ equals twice the socle degree. 
We record the following general facts about the canonical monomial ${\bf x^K}$.

\begin{corollary} \label{cor:xK}
The canonical monomial of a Riemann-Roch monomial ideal $M$ satisfies
$$ {\rm degree}({\bf x^K}) \,=\, 2 \cdot {\rm genus}(M) - 2 \quad
\hbox{and} \quad {\rm rank}({\bf x^K}) \,= \,{\rm genus}(M)- 2. $$
\end{corollary}

Experts will note that the rank is off by one when compared
to the canonical divisor of an algebraic curve or metric graph.
This discrepancy will be addressed in Remark \ref{rmk:recon} below.
We now prepare for an analogue of Clifford's theorem on special divisors.

\begin{lemma}\label{revtri_lem}
The rank is superadditive for monomials  ${\bf x^{a}}$ and ${\bf x^{b}}$ of non-negative rank:
\begin{equation}
\label{eq:superadd}
{\rm rank}({\bf x^a} \cdot {\bf x^b}) \,\, \geq \,\, {\rm rank}({\bf x^a})+{\rm rank}({\bf x^b}). 
\end{equation}
\end{lemma}

\begin{proof}
Consider an arbitrary monomial ${\bf x^c}$ of degree at most ${\rm rank}({\bf x^a})+{\rm rank}({\bf x^b})$ 
such that ${\bf x^c}$ divides ${\bf x^a} \cdot {\bf x^b}$. The following formulas define
monomials ${\bf x^{c'}}$ and ${\bf x^{c''}}$ such that
${\bf x^{c'}}$ divides ${\bf x^a}$ and ${\bf x^{c''}}$ divides ${\bf x^b}$:
$$
c'_i \, =\,
\begin{cases} 
a_i & \text{if} ~c_i \geq a_i,\\ 
c_i & \text{otherwise},
\end{cases} \quad
\hbox{and} \qquad
c''_i \, =\,
\begin{cases}
c_i-a_i &\text{if}~c_i \geq a_i ,\\
0 & \text{otherwise}.
\end{cases}  
$$
By construction, ${\bf x^{c'}} \cdot {\bf x^{c''}}={\bf x^c}$. This implies that either 
${\rm degree}({\bf x^{c'}}) \leq {\rm rank}({\bf x^a})$ or ${\rm degree}({\bf x}^{c''}) \leq {\rm rank}({\bf x^b})$. In other words, either ${\bf x^a/x^{c'}}$ or ${\bf x^b/x^{c''}}$ is in $M$, and hence, their product
${\bf \frac{x^a \cdot x^b}{x^c}}$ is also in $M$. 
From this we infer the inequality (\ref{eq:superadd}) as follows:
Since ${\bf x^c}$ is an arbitrary monomial of degree less than or equal to
${\rm rank}({\bf x^a})+{\rm rank}({\bf x^b})$, and ${\bf x^c}$ divides ${\bf x^a} \cdot {\bf x^b}$,
we know that any monomial 
that ``defines'' the rank of ${\bf x^a} \cdot {\bf x^b}$  
\ (i.e.,  a monomial ${\bf x^d}$  of minimum degree
such that ${\bf x^d}$ divides ${\bf x^a} \cdot {\bf x^b}$ 
and $\frac{{\bf x^a} \cdot {\bf x^b}}{{\bf x^d}} \not\in M$) \
has degree strictly greater than
${\rm rank}({\bf x^a})+{\rm rank}({\bf x^b})$. Hence, 
${\rm rank}({\bf x^a} \cdot {\bf x^b}) \geq {\rm rank}({\bf x^a})+{\rm rank}({\bf x^b})$.
\end{proof}

\begin{corollary}{\rm ({\bf Clifford's Theorem})}
Let ${\bf x^b}$ be a monomial dividing ${\bf x^K}$
such that both ${\rm rank}({\bf x^{b}})$ and ${\rm rank}({\bf x^{K}/x^{b}})$ are non-negative. Then $\,{\rm rank}({\bf x^{b}}) \leq ({\rm degree}({\bf x^{b}})-1)/2$.
\end{corollary}

\begin{proof}
Lemma \ref{revtri_lem} and Corollary \ref{cor:xK} imply
$$ {\rm rank}({\bf x^b})+{\rm rank}({\bf x^{K}/x^{b}}) \,\leq\, {\rm rank}({\bf x^K})\,=\,{\rm genus}(M)-2 .$$
From the Riemann-Roch formula we have 
$$ {\rm rank}({\bf x^{b}})-{\rm rank}({\bf x^{K}/x^{b}})\,\,=\,\,{\rm degree}({\bf x^b})-({\rm genus}(M)-1). $$
The desired conclusion follows by adding these two identities and dividing by $2$.
\end{proof}

The construction of all Riemann-Roch monomial ideals of genus $g$ works as follows.
We first fix a monomial $\bf{x^K}$ with degree $2g-2$. 
Next we choose a set $\mathcal{M}$ of monomials of degree $g-1$ that divide ${\bf x^K}$.
Then there exists a unique artinian monomial ideal $M$ whose
socle is spanned by the monomials in  $\mathcal{M}$ and their
complements relative to ${\bf x^K}$:
\begin{equation}
\label{eq:buildRR}
 {\rm MonSoc} \quad = \quad \mathcal{M} \,\cup\, \{ \,{\bf x^{K}/x^{b}} \,: \,{\bf x^b} \in \mathcal{M \,} \}.
 \end{equation}
Namely, the ideal $M$ is the intersection of the  irreducible ideals 
$\langle x_1^{c_1+1},\dots,x_m^{c_m+1} \rangle $
where  ${\bf x^c}$ runs over the set ${\rm MonSoc}$.
Then $M$ is artinian, level, and reflection-invariant.

We shall now make the connection to the Riemann-Roch theorem for graphs.
As in Section 2, we let $G$ denote a saturated
graph on $n$ nodes, with $u_{ij} > 0$ edges between nodes $i$ and $j$, and
$M_G$ the initial monomial ideal in $\mathbb{K}[x_1,\ldots,x_{n-1}]$
of the toppling ideal $I_G$ with respect to
a reverse lexicographic term order having $x_n$ as smallest variable.

\begin{theorem} \label{thm:isRR} Let $G$ be a saturated graph with $n$ vertices,
$e$ edges,  and node $ i$ having degree $d_i$.
Then the monomial ideal $M_G$
is  Riemann-Roch with canonical monomial
\begin{equation}
\label{eq:m-n+2}
 \quad {\bf x^{K}} \,=\, \prod_{i=1}^{n-1}{x_i}^{d_i+u_{in}-2} 
\quad {\rm and} \quad {\rm genus}(M_G) \,=\, e-n+2 . 
\end{equation}
  \end{theorem}

\begin{proof}
The monomial ideal $M_G$ is artinian, and it is level 
because all the socle monomials $s_{\mathcal{T}}$ 
in Corollary \ref{gensocle_theo} have the same degree $e-n+1$.
This quantity is the cyclotomic number (or genus) of the graph $G$, which, by \cite{BCT},
coincides with the common degree of all maximal parking functions. 
There is a natural involution $\phi$ on the set of $(n-1)!$ maximal flags
$\mathcal{T}$ of subsets in $[n-1]$.  It takes a flag
$\mathcal{T}:T_1 \subset T_2 \subset \dots \subset T_{n-2} \subset T_{n-1}$ 
to the reverse flag $\phi(\mathcal{T}):T_{n-1}\backslash T_{n-2} \subset T_{n-1} \backslash T_{n-3} 
\subset \cdots \subset T_{n-1} \backslash T_1 \subset T_{n-1}$.
Using the identification between flags and socle monomials 
in Corollary \ref{gensocle_theo}, we have
\begin{equation}
\label{eq:involution}
 s_{\phi(\mathcal{T})} \,\,=\,\, {\bf x^K}/s_{\mathcal{T}}, 
 \end{equation}
where ${\bf x^K}$ is the monomial defined  in (\ref{eq:m-n+2}). Hence, $M_G$ is also reflection-invariant.
\end{proof}

\begin{remark} \rm
Not every Riemann-Roch monomial ideal arises as an initial monomial ideal $M_G$ for a connected graph $G$.
To see this, note that the number of socle monomials of $M_G$ is at most $m!$ where $m$ is the number of variables of $M_G$. On the other hand, 
  Riemann-Roch monomial ideals can in general   have more than $m!$ socle monomials.  Furthermore, the initial monomial ideal $M_G$ for a connected graph $G$
 is not necessarily Riemann-Roch. To see this, consider the four-cycle $C_4$. Its initial ideal is 
 $M_{C_4}=\langle x_1,x_2,x_3 \rangle^2$, with socle monomials $x_1$, $x_2$ and 
 $x_3$, and $M_{C_4}$ is not reflection-invariant. \qed \end{remark}

In the remainder of this section we show how
the familiar Riemann-Roch theorem for graphs is derived from Theorem \ref{thm:isRR}.
While the proof still assumes that $G$ is saturated, that hypothesis
will be removed in the next section. The following algebraic definitions
are valid for any undirected connected graph $G$ on $[n]$.
Here $G$ need not be saturated.

The Laurent polynomial ring 
$\mathbb{K}[{\bf x}^{\pm 1}] = \mathbb{K}[x_1^{\pm 1}, \ldots, x_n^{\pm 1}]$
is a module over the polynomial ring $\mathbb{K}[{\bf x}] = \mathbb{K}[x_1, \ldots, x_n]$.
The Laplacian lattice $\,{\rm image}_\mathbb{Z}(\Lambda_G)\,$
is a sublattice of rank $n-1$ in $\mathbb{Z}^n$,
and we write $L_G$ for the corresponding {\em lattice module},
as in \cite[Definition 9.11]{MilStr05}. Thus, $L_G$ is the
$\mathbb{K}[{\bf x}]$-submodule of $\mathbb{K}[{\bf x}^{\pm 1}] $
generated by all Laurent monomials
${\bf x^w}$ that have degree zero in the grading by
the group ${\rm Div}(G)$. If $G$ is saturated
then $L_G$ is generic and the Scarf complex
is a minimal free resolution by \cite[Theorem 9.24]{MilStr05}.
That Scarf complex is precisely the {\em Delaunay triangulation}
 in  \cite{AmiMan10}, and our point here is to redevelop the  Amini-Manjunath approach in the language of commutative algebra.

Consider the set of all Laurent monomials ${\bf x^c}$ that are in the
socle of the module $L_G$:
$$\,{\rm MonSoc}(L_G) \,= \, \bigl\{{\bf x^c} \not\in L_G \,|\,\,x_i {\bf x^c} \in L_G
\,\, \forall i  \,\bigr\}. $$
This socle is a set of Laurent monomials
on which the lattice $L_G$ acts with finitely many orbits,
so the computation of ${\rm MonSoc}(L_G)$ is a finite algorithmic problem,
as in \cite[\S 9.3]{MilStr05}.
The problem's solution is given by the
socle monomials   (\ref{eq:MGsoc})
of our monomial ideal $M_G$.

\begin{lemma}\label{Laplatsoc_theo}
The socle monomials of the lattice module $L_G$ are precisely  of the form
$\,s_{\mathcal{T}} \cdot {\bf x^{w}}/x_n\,$,  where  $\,s_{\mathcal{T}} \in {\rm MonSoc}(M_G)\,$
and ${\bf x^{w}}$ runs over the minimal generators of $L_G$. 
\end{lemma}

\begin{proof}
Since $G$ is assumed to be saturated throughout this section, the lattice module $L_G$ is generic in the
sense of \cite[Definition 9.23]{MilStr05},
with $M_G$ being the reverse lexicographic initial ideal of the corresponding
lattice ideal $I_G$.
We claim that the stated characterization of the socle is valid
for any generic lattice module that is artinian. Indeed, by the proof of
 \cite[Theorem 5.2]{PeevaStr98},
the $\mathbb{Z}^n$-degrees of the $n$-th syzygies of $L_G$ are
the vectors ${\bf u}+{\bf w}$, where ${\bf u} $ runs over the $\mathbb{Z}^{n-1}$-degrees
of the $(n-1)$-st syzygies of $M_G$ and ${\bf w}$ is any vector in the lattice.
The socle degrees of $M_G$ are the vectors 
${\bf u}-e_1 - \cdots - e_{n-1}$ in $\mathbb{Z}^{n-1}$,
and the socle degrees of $L_G$ are the vectors
${\bf u}+{\bf w} - e_1 \cdots - e_{n-1} - e_n$ in $ \mathbb{Z}^n$.
\end{proof}

We now identify Laurent monomials
${\bf x^u}$ with divisors on the graph $G$.
The $i$-th coordinate $u_i$ of the exponent vector
${\bf u}$ is the multiplicity of
node $i$ in the divisor ${\bf x^u}$.
The {\em degree} of the divisor ${\bf x^u}$ is
its total degree as a monomial,
$\,{\rm degree}({\bf x^u}) = u_1 + \cdots + u_n$. 
The {\em rank} of the divisor ${\bf x^u}$ is 
defined by the same formula (\ref{rank_form}) as
 in Lemma \ref{rankformart_theo}:
\begin{equation} \label{rank_form2}
                 {\rm rank}({\bf x^u}) \quad =\,\min_{{\bf x^c} \in {\rm MonSoc}(L_G)}{\rm degree}^{+}({\bf u-c})-1.
\end{equation}
Thus, ${\rm rank}({\bf x^u}) \geq 0 $ if and only if ${\bf x^u}$ lies in $L_G$.  Our definition of rank of the divisor ${\bf x^u}$ coincides with the rank of ${\bf u}$ as in \cite{BakNor07}. To see this, use Lemma 2.7 in \cite{BakNor07} and note that  the exponents of the socle Laurent monomials of $L_G$ are  the elements of the set $\mathcal{N}$ in~\cite{BakNor07}.

We finally define the {\em canonical divisor} of $G$ to be the monomial
$$ {\bf x^k} \,\, = \,\, x_1^{d_1-2} x_2^{d_2-2} \cdots \, x_n^{d_n-2}, $$
where $d_i = \sum_{j\not=i} u_{ij}$ is the degree of node $i$.
Finally, we recall that the genus of  $G$ is $e-n+1$, where $e$ is the number of edges.
The following is precisely \cite[Theorem 1.12]{BakNor07}:

\begin{theorem}[Baker-Norine] \label{thm:BN}
Riemann-Roch holds for any divisor ${\bf x^u}$ on the graph $G$:
 \begin{equation} \label{eq:RRF2}
{\rm rank}({\bf x^{u}})-{\rm rank}({\bf x^k}/{\bf x^{u}})\,\,=\,\,{\rm degree}({\bf x^{u}})- {\rm genus}(G) +1.
\end{equation}
\end{theorem} 

\begin{proof}
It follows from Lemma \ref{Laplatsoc_theo} that all socle monomials of $L_G$ have degree equal to the genus of $G$ minus one.
 The lattice module $L_G$ is also reflection-invariant, in the sense that
$ {\bf x^u} \in  {\rm {MonSoc}}(L_G)$ implies 
 $\,{\bf x}^{\bf k}/{\bf x^u} \in {\rm {MonSoc}}(L_G)$.
 Using the representation in Lemma~\ref{Laplatsoc_theo}, 
 the resulting involution $\phi$ on  ${\rm {MonSoc}}(L_G)$
 can be written as follows:
 \begin{equation} \phi((s_{\mathcal{T}}/x_n)\cdot {\bf x^{w}}) \quad 
\, =\, \quad (s_{\phi(\mathcal{T})}/x_n)\cdot {\bf x^{-w}} \cdot \left(\frac{x_n^{d_n}}{\prod_{i=1}^{n-1}x_i^{u_{in}}}  \right),
\end{equation} 
where $\phi(\mathcal{T})$ denotes the reverse flag as in (\ref{eq:involution}). Note that the image of $\phi$ is in ${\rm {MonSoc}}(L_G)$, since  $\,x_n^{d_n}/ \prod_{i=1}^{n-1}x_i^{u_{in}}$ is in $  L_G$.
The proof of Theorem \ref{thm:BN} is now entirely analogous to that
of Theorem~\ref{RR_theo}. In other words, 
our argument 
for the validity of the Riemann-Roch 
formula for reflection-invariant artinian level monomial ideals
generalizes in a straightforward manner 
to reflection-invariant artinian level lattice modules.
\end{proof}

\begin{remark} \label{rmk:recon} \rm
The rank of the canonical monomial of $M_G$ equals
the rank of the canonical divisor of the graph $G$, but the degree of the former
is two more than that  of the latter. \qed
\end{remark}

\section{Non-Saturated Graphs}

We turn to graphs $G$ that are not necessarily saturated.
The binomials in (\ref{eq:IJbinomial}), their syzygies in (\ref{eq:cycres2}),
and the ideals $I_G$ and $M_G$ are still well-defined.
However, the minimality in Theorem \ref{satminres_theo}
is no longer true, and the choice of
term order is more subtle, as the next example~shows.

\begin{example} \rm
Let $G$ be the edge graph of a triangular prism, labeled so that $I_G$ equals
$$[\langle
a^3-bcd, b^3-ace,c^3-abf, 
d^3-aef, 
e^3-bdf, f^3-cde
\rangle: \langle abcdef \rangle^{\infty}].$$
This toppling ideal has $22$ minimal generators, and its  free resolution has
 Betti numbers  $(1,22,92,147,102,26)$. The same holds for the 
 ideal that represents  parking~functions:
  $$ \begin{matrix} M_G \,\,\, = \,\,\, {\rm in}(I_G) \,\, \, = 
& \bigl\langle\, a^3, a b c,a c e^2,b^3, b^2 e^2,c^3,   c d e, d^3,d^2 e^2,  e^3 \,\,
 a^2 b^2,a^2 c^2,    a^2 d^2, \\  &  \quad a b d e, b^2 c^2, b c d^2,  \,\,
 a^2 b e^2,  a^2 d e^2,a b^2 d^2,a c^2 d^2,   b^2 d^2 e,  b c^2 e^2
\,\bigr\rangle .\\
\end{matrix}
$$
Here the reverse lexicographic term order with 
$a {>} b {>} c {>} d {>} e {>} f$ was used.
However, if we take 
$e{>}d{>}c{>}b{>}a{>}f$ then 
the reverse lexicographic
   Gr\"obner basis requires two more binomials.
 Now, the initial monomial ideal has
  Betti numbers  $(1,24,98,153,104,26)$. 
\qed
\end{example}

To explain the phenomenon in this example,
we fix a spanning tree $T$ of the graph $G$ that is rooted at the node $n$,
and we order the unknowns according to a linear extension of $T$.
  Thus, we fix an ordering of $[n]$ such that $i>j$ if the node $i$ is a descendant of the node $j$ in $T$.
  A term order on $\mathbb{K}[{\bf x}]$
  is a {\em spanning tree order} if it is a reverse
    lexicographic term order whose variable ordering is compatible
    with some spanning tree rooted at $n$.
    One spanning tree order is the toppling order considered in \cite[Theorem 10]{CorRosSal02}. 
    See also \cite[\S 5]{Perkinson} for a discussion of Gr\"obner bases of toppling ideals
    in the inhomogeneous case.
    
\begin{theorem}\label{nonsatinit_theo}
The toppling ideal $I_G$ is generated by the binomials
${\bf x}^{I \rightarrow J} - {\bf x}^{J \rightarrow I} $
where $(I,J)$ runs over splits of $[n]$ such that
the  subgraphs of $G$ induced on $I$ and $J$ are connected.
For any spanning tree order,  these binomials form a Gr\"obner basis 
of $I_G$ with initial monomial ideal $M_G$. 
The complexes constructed in (\ref{eq:cycres1})
are cellular free resolutions.
\end{theorem}

\begin{proof}
The first paragraph in the proof of Theorem \ref{satminres_theo} is valid in the
non-saturated case. It shows that the binomials ${\bf x}^{I \rightarrow J} - {\bf x}^{J \rightarrow I} $
lie in $I_G$. For the spanning tree term order, the leading monomials
are ${\bf x}^{I \rightarrow J}$, where $n \in J \backslash I$, and hence
the initial ideal ${\rm in}(I_G)$ contains
the monomial ideal $M_G$ of (\ref{eq:MG}).
Again, both ideals are artinian of the same colength in
 $\mathbb{K}[{\bf x}_{\backslash n}] $, and hence they are equal.
 This establishes the Gr\"obner basis property.
The argument in the proof of  \cite[Theorem 14]{CorRosSal02}
shows that the property that the subgraphs of $G$ induced on $I$ and $J$
are connected
characterizes a minimal Gr\"obner basis of $I_G$ and the
minimal generators of $M_G$. In particular, these binomials
${\bf x}^{I \rightarrow J} - {\bf x}^{J \rightarrow I} $ generate $I_G$.

Our last assertion states that 
(\ref{eq:cycres1}) with differentials (\ref{eq:cycres2})
gives a free resolution of $I_G$,
and dropping the last term in  (\ref{eq:cycres2}) 
gives a free resolution of $M_G$.
This claim is proved by deformation to
generic monomial modules, as explained in
\cite[\S 6.2]{MilStr05}. To be precise, in our situation
we replace $G$ by a nearby saturated graph $G_\epsilon$
with fractional edge numbers $u_{ij}(\epsilon)$ between any pair of nodes.
The monomial ideal $M_{G_\epsilon}$ is generic and degenerates to $M_G$.
The lattice ideals $I_G$ and $I_{G_\epsilon}$ are represented 
by the corresponding lattice modules $L_G$ and $L_{G_\epsilon}$.
These are submodules of the Laurent polynomial ring 
as in  \cite[Definition 9.11]{MilStr05}. The lattice module $L_{G_\epsilon}$ is generic and degenerates to $L_{G}$.
According to \cite[Theorem 6.24]{MilStr05},
the Scarf complex of $M_{G_\epsilon}$ {\em with labels from $G$}
gives a free resolution of $M_G$.
 Likewise, the Scarf complex of the generic lattice module $L_{G_\epsilon}$ {\em with labels from $G$}
gives a free resolution of $L_G$.
The resulting minimal free resolution of $I_{G_\epsilon}$ 
degenerates to 
a (typically non-minimal) resolution of $I_G$, using
 \cite[Corollary 9.18]{MilStr05}.
These free resolutions of $M_G$ and $I_G$ are
cellular because they are given by labeled simplicial complexes.
\end{proof}

Here is an example that illustrates the degeneration used in the proof above.

\begin{example} \rm
Let $n=4$ and $G$ be the $4$-cycle $1 - 2 - 3 - 4 - 1 $. For 
$\delta, \epsilon \in \mathbb{N} $ consider the graph $G_{\delta,\epsilon}$ 
that has $\delta$ edges for every edge in $G$ and
$\epsilon $ edges for every non-edge of $G$.
Then $G_{\delta,\epsilon}$ is saturated for $\delta,\epsilon > 0$. The
Scarf complex in Example \ref{ex:ressi} gives the minimal free resolution
of   $M_{G_{\delta,\epsilon}} = {\rm in}(I_{G_{\delta,\epsilon}})$
and this lifts to the minimal free resolution of $I_{G_{\delta,\epsilon}}$. By
Theorem \ref{thm:isRR}, the monomial ideal $M_{G_{\delta,\epsilon}}$ 
is Riemann-Roch, its genus is $\,4 \delta + 2 \epsilon - 2 $, and its canonical monomial equals
$\,{\bf x^K} = x_1^{3 \delta + \epsilon - 2} x_2^{2 \delta + 2\epsilon -2} x_3^{3 \delta + \epsilon -2 }$.
The involution ${\bf x^b} \mapsto {\bf x}^{{\bf K}-{\bf b}}$
on its six socle monomials is given by swapping the two rows below:
\begin{equation}
 {\rm MonSoc}(M_{G_{\delta,\epsilon}}) \, = \,
\left\{
\begin{matrix} \,
x_1^{\delta-1} x_2^{\delta+\epsilon-1} x_3^{2 \delta+\epsilon-1}, & \!\!
x_1^{\delta-1} x_2^{2 \delta+\epsilon-1} x_3^{\delta+\epsilon-1},& \!\!
x_1^{\delta+\epsilon-1} x_2^{2\delta+\epsilon-1} x_3^{\delta-1} \, \\ 
x_1^{2 \delta+\epsilon-1} x_2^{\delta+\epsilon-1} x_3^{\delta-1}, &  \!\!
x_1^{2 \delta+\epsilon-1} x_2^{\epsilon-1} x_3^{2 \delta-1} , &\!\!
x_1^{2 \delta-1} x_2^{\epsilon-1} x_3^{2\delta+\epsilon-1} 
\end{matrix} \!
\right\}.
\end{equation}
Setting $\delta = 1$ and $\epsilon = 0$, we get the 
parking function monomial ideal of the $4$-cycle
$$ M_{G_{1,0}} \, = \, M_G \, = \, \langle x_1,x_2,x_3 \rangle^2 \,= \,
\langle x_1^2, x_1 x_2, x_1 x_3, x_2^2, x_2 x_3, x_3^2 \rangle .$$
Here, ${\rm MonSoc}(M_G) =  \{x_1,x_2,x_3\}$. This ideal is
not reflection-invariant and hence not Riemann-Roch.
The cellular resolution of $M_G$ induced from $M_{G_{\delta,\epsilon}}$
is not minimal.
\qed
\end{example}

We now take a closer look at the 
combinatorial structure of our resolutions.
Let ${\rm Bary}(G)$ denote the first barycentric subdivision of
the $(n-2)$-simplex, whose $2^{n-1}-1$ vertices, namely the non-empty subsets $I$ of $[n-1]$,
are labeled by the corresponding monomials ${\bf x}^{I \rightarrow [n]\backslash I}$.
Thus, ${\rm Bary}(G)$ is the cellular free resolution of $M_G = {\rm in}(I_G)$
referred to in Theorem  \ref{nonsatinit_theo}.
Each simplex in ${\rm Bary}(G)$ is labeled by the least common
multiple of the monomials that label its vertices. For any ${\bf c} \in \mathbb{N}^{n-1}$ we write
${\rm Bary}(G)_{\prec {\bf c}}$ for the subcomplex
consisting of all simplices in ${\rm Bary}(G)$ whose labels properly divide ${\bf x^c}$.

\begin{corollary} \label{cor:minsyz1}
The number of minimal $i$-th syzygies of the monomial ideal $M_G$ in degree ${\bf c}$
is equal to the rank of the reduced homology group
$\tilde{H}_{i-1}({\rm Bary}(G)_{\prec {\bf c}};\mathbb{K})$.
\end{corollary}

\begin{proof}
This follows immediately from Theorem  \ref{nonsatinit_theo}
and \cite[Theorem 4.7]{MilStr05}.
\end{proof}

We next state the analogous result for the lattice module $L_G$,
that is, the $\mathbb{K}[{\bf x}]$-module
generated by all Laurent monomials
whose exponent vector lies in the Laplacian lattice 
${\rm image}_\mathbb{Z}(\Lambda_G)$.
We identify this lattice with 
$\, \mathbb{Z}^n/\mathbb{Z}{\bf e}\,$ by writing its elements
as $\Lambda_G \cdot {\bf v}$ where
each ${\bf v} \in \mathbb{Z}^n$ is unique modulo
$\,\mathbb{Z}{\bf e}\, = {\rm ker}(\Lambda_G)$.
The {\em tropical metric} on
$\, \mathbb{Z}^n/\mathbb{Z}{\bf e}\,$ is 
$$ {\rm dist}({\bf u},{\bf v}) \,\,= \,\,
{\rm max} \bigl\{\, | u_i + v_j - u_j - v_i | \,: 1 \leq i < j \leq n \bigr\}. $$
We write ${\rm Apt}(G)$ for the corresponding {\em flag simplicial complex}.
Thus, ${\rm Apt}(G)$ is the simplicial complex 
whose simplices are
subsets $S$ of $\, \mathbb{Z}^n/\mathbb{Z}{\bf e}\,$  such that
$ {\rm dist}({\bf u},{\bf v} ) \leq 1$ for ${\bf u}, {\bf v} \in S$.
The notation ``{\rm Apt}'' refers to the fact that this infinite simplicial complex
is the standard {\em apartment} in the affine building of Lie type $A_{n-1}$.
It is well-known that ${\rm Apt}(G)$ is pure of dimension $n-1$
and that it  triangulates the $(n-1)$-dimensional affine space $ \mathbb{R}^n/\mathbb{R}{\bf e}$.
For more on buildings and their connection to tropical geometry, see \cite{albania}.

The apartment ${\rm Apt}(G)$ is precisely the same
as the Delaunay triangulation constructed in 
\cite{AmiMan10}, and it also coincides with the
Scarf complex of $G_\epsilon$ that we 
used to give a cellular resolution of $L_G$.
The number of $i$-faces of ${\rm Apt}(G)$ modulo the
lattice action is given by (\ref{eq:cycnumber}).
Each vertex ${\bf v}$ of ${\rm Apt}(G)$ is labeled by
the corresponding Laurent monomial ${\bf x}^{\Lambda_G {\bf v}}$,
and each face is labeled by the least common multiple of its vertex labels.
Thus, each face of ${\rm Apt}(G)$ is labeled by a Laurent monomial
of degree $\geq 0$. We write
${\rm Apt}(G)_{\prec {\bf c}}$ for the subcomplex
of all simplices  whose label properly divides ${\bf x^c}$.

\begin{corollary} \label{cor:minsyz2}
The number of minimal $(i+1)$-st syzygies of the lattice module $L_G$ in degree ${\bf c}$
is  the rank of the reduced homology 
$\tilde{H}_{i}({\rm Apt}(G)_{\prec {\bf c}};\mathbb{K})$.
The sum of these ranks over all ${\bf c}$ modulo
${\rm image}_{\mathbb{Z}}(\Lambda_G)$
 counts the minimal $i$-th syzygies of the toppling ideal $I_G$.
\end{corollary}

We conjecture that the ranks of the homology groups in the two corollaries coincide.

\begin{conjecture} \label{conj:betti}
The Betti numbers of the toppling ideal $I_G$ do not increase
when passing to the initial ideal $M_G$. More precisely,
for all $i \geq 0$ and all ${\bf c} \in \mathbb{N}^{n-1}$, we have
\begin{equation}
\label{eq:baryapt}
 \tilde{H}_{i-1}({\rm Bary}(G)_{\prec {\bf c}};\mathbb{K}) \,\,\, \simeq \,\,\,
\tilde{H}_{i}({\rm Apt}(G)_{\prec {\bf c}};\mathbb{K}).
\end{equation}
\end{conjecture}

This conjecture has been verified for many graphs using the
software {\tt Macaulay2}. We note that the two simplicial complexes
appearing in (\ref{eq:baryapt}) are different from the
 complex $\Delta_D$ used in Hochster's formula
 for the Betti numbers of a lattice ideal
  \cite[Theorem 7.4]{Perkinson}.

\begin{example} \rm
The simplicial complexes ${\rm Apt}(G)_{\prec {\bf c}}$
can be large even for small graphs. 
Let $G$ be the graph on four nodes, labeled
$a,b,c,d$, with Laplacian matrix
$$
\Lambda_G \,= \,
\begin{pmatrix}
u_{12} {+}u_{13} {+} u_{14} \!\!\,\,  & -u_{12} & -u_{13} & -u_{14} \\
 -u_{12} & \!\!\!\! u_{12} {+} u_{23} {+} u_{24} \!\!\!\! & -u_{23} & -u_{24} \\
 -u_{13} & -u_{23} & \!\!\! \! u_{13}{+} u_{23}{+} u_{34} \!\!\!\! & -u_{34} \\
 -u_{14} & -u_{24} & -u_{34} & \!\!\!\! u_{14} {+}u_{24} {+} u_{34} 
 \end{pmatrix}
 \, = \,
 \begin{pmatrix}
\phantom{-}2  & -2  & \phantom{-}0 &\phantom{-} 0 \, \\
-2 &\phantom{-} 3 & -1 & \phantom{-}0 \,\\
\phantom{-}0 & -1 & \phantom{-}4 & -3 \,\\
\phantom{-}0 & \phantom{-}0   &-3 & \phantom{-}3 \,
\end{pmatrix}.
 $$
Both the toppling ideal and the ideal of parking functions are complete intersections:
 $$ I_G \,= \,\langle a^2-b^2, b-c, c^3-d^3 \rangle
 \quad \hbox{and} \quad
 M_G \, = \,\langle a^2, b, c^3 \rangle. $$
 The monomial ideal $M_G$ has one minimal first syzygy in degree
 ${\bf c} = (2,0,3,0)$. The simplicial complex $\,{\rm Bary}(G)_{\prec {\bf c}}\,$
 consists of two isolated nodes $ a^2$ and $ c^3$. The simplicial complex 
 $\,{\rm Apt}(G)_{\prec {\bf c}}\,$ is two-dimensional  but
 it has the homology of a circle. It consists of
 $12$ triangles,   $ 28$ edges and $16 $ vertices,
  labeled by the following generators of~$L_G$:
$$ 
1 ,\,
\frac{ c }{ b},\,
\frac{ c^2 }{ b^2},\,
\frac{ c^3 }{ b^3},\,
\frac{ c^2 }{ a^2},\,
\frac{ c^3 }{ a^2b} , \,
\frac{ c^3 }{ d^3},\,
\frac{ a^2c }{ d^3 },\,
\frac{ a^2 }{ c^2},\,
\frac{ a^2c^2 }{ bd^3},\,
\frac{ a^2 }{ bc},\,
\frac{ a^2c^3 }{ b^2d^3},\,
\frac{ a^2 }{ b^2},\,
\frac{ a^2c }{ b^3},\,
\frac{ a^2 c^2 }{ b^4},\,
\frac{ a^2c^3 }{ b^5}.
$$
The lattice module $L_G$ has one second syzygy in this degree,
translating into a first syzygy of $I_G$.
It is represented in
$\,{\rm Apt}(G)_{\prec {\bf c}}\,$
by the $4$-cycle 
$\, 1  ,a^2 /b^2,a^2c^3/b^2d^3,    c^3 / d^3$. \qed
\end{example}

 At present, no explicit minimal free resolution of $M_G$ is known.
 Finding such a resolution was stated as an open problem 
 by Postnikov and Shapiro in \cite[\S 6]{PosSha}. We do not even know whether the Betti numbers depend on the characteristic of the field~$\mathbb{K}$.

An explicit  formula for the Betti numbers of the toppling ideal $I_G$
was conjectured by Wilmes in \cite{Wil10}. See also \cite[\S 7.4]{Perkinson}.
Wilmes' formula is combinatorial, and it has been verified for all graphs
with $n \leq 6$ nodes. At present we do not know how to relate
Wilmes' conjecture to the ranks of the homology groups in Corollaries
\ref{cor:minsyz1} and \ref{cor:minsyz2}.

It is known, thanks to \cite[Theorem 3.10]{Wil10}, that Conjecture \ref{conj:betti}
is true for the maximal syzygies, with index $i = n-2$.
We have the following combinatorial characterization:

\begin{corollary}\label{MaxSyz_cor}
The maximal syzygies of the parking function ideal $M_G$, or of the
 toppling ideal $I_G$, are in bijection with the
acyclic orientations of $G$ with node $n$ as unique~sink.
\end{corollary}

See \cite[Theorem 7.6]{Perkinson} for an alternative
but equivalent formulation of this result.

\begin{proof}
The monomial ideal $M_G$ is artinian, so its maximal syzygies
correspond to the socle elements. These are the maximal parking functions,
and, by \cite[Theorem 4.1]{BCT}, they correspond to 
cyclic orientations of $G$ with node $n$ as unique~sink.
Since all maximal syzygies of $M_G = {\rm in}(I_G)$ lift to maximal syzygies of $I_G$,
the same result holds for $I_G$.
\end{proof}

We now derive the Riemann-Roch theorem for non-saturated graphs $G$.
  Let $M_G$ be the initial ideal with respect to a spanning tree order on the variables 
with $x_n$ as the least. By Corollary \ref{MaxSyz_cor}, 
we know that  the socle monomials of the Laurent monomial module $L_G$ are $s\cdot x_{n}^{-1}\cdot {\bf x^w}$ where $s$ runs over all socle monomials of $M_G$ and ${\bf x^{w}}$ runs over minimal generators of $L_G$. Unlike in the saturated case, the monomial ideal $M_G$ is generally not reflection-invariant.
But the Laurent monomial module $L_G$ is always reflection-invariant. To see this, we use 
Lemma 3.2 of \cite{BakNor07} to deduce that $s_{\mathcal{T}}/x_n$, defined in (\ref{eq:MGsoc}), is not 
contained in $L_G$ for any complete flag $\mathcal{T}$ of $[n]$. This implies that $s_{\mathcal{T}}/{x_n}$ is a 
socle element of $L_G$, since every Laurent monomial of degree greater than ${\rm degree}(s_{\mathcal{T}}/{x_n})={\rm genus}(G)-1$ is  in $L_G$. We now immediately verify that $L_G$ is reflection-invariant with the involution on ${\rm MonSoc}(L_G)$ that takes $\,s_{\mathcal{T}}{\bf x^w}/{x_n}\,$ to 
$\, s_{\phi{(\mathcal{T})}}{\bf x^{-w}} x_n^{d_n-1}/\prod_{i=1}^{n-1}x_i^{u_{in}}$, where $\phi(\mathcal{T})$ is the reverse flag of $\mathcal{T}$ exactly as in the proof of Theorem \ref{thm:BN}.
The canonical monomial is
 $$ {\bf x^k} \,\, = \,\, x_1^{d_1-2} x_2^{d_2-2} \cdots \, x_n^{d_n-2}, $$ 
 where $d_i = \sum_{j\not=i} u_{ij}$ is the degree of node $i$. Hence, 
 $L_G$ satisfies the Riemann-Roch formula,  in its monomial formulation 
  (\ref{eq:RRF}), with $M = L_G\,$ and $\,{\rm genus}(M) = {\rm genus}(G)$.

 \bigskip \medskip
 
 \noindent
 {\bf Acknowledgments}:
  We thank David Perkinson,
  B.V.~Raghavendra Rao,  Frank-Olaf Schreyer,
  and John Wilmes for helpful discussions.
      Bernd Sturmfels was partially supported by
  the U.S.~National Science Foundation
  (DMS-0757207 and DMS-0968882).

\bigskip
\medskip

\begin{small}
\noindent
Madhusudan Manjunath, Fachrichtung Mathematik,\\
Universit\"at des Saarlandes,
Saarbr\"ucken, Germany. \\
{\tt manjun@mpi-inf.mpg.de}

\medskip

\noindent
Bernd Sturmfels, Department of Mathematics, \\ 
University of California, Berkeley, USA, \\
{\tt bernd@math.berkeley.edu}
\end{small}

\end{document}